\documentclass[12pt]{amsart}
\usepackage{amsmath,amsthm,amsfonts,amssymb}

\newtheorem{metatheorem}{metatheorem}[section]
\newtheorem{theorem}[metatheorem]{Theorem}

\newtheorem{corollary}[metatheorem]{Corollary}

\newtheorem{itexample}[metatheorem]{Example}
\newtheorem{itdefinition}[metatheorem]{Definition}
\newtheorem{itremark}[metatheorem]{Remark}
\newtheorem{itquestion}[metatheorem]{Question}
\newtheorem{itproblem}[metatheorem]{Problem}
\newtheorem{italgorithm}[metatheorem]{Algorithm}

\newenvironment{example}%
    {\begin{itexample}\begin{rm}}{\end{rm}\end{itexample}}
\newenvironment{definition}%
    {\begin{itdefinition}\begin{rm}}{\end{rm}\end{itdefinition}}
\newenvironment{remark}%
    {\begin{itremark}\begin{rm}}{\end{rm}\end{itremark}}
    {\begin{itquestion}\begin{rm}}{\end{rm}\end{itquestion}}
    {\begin{itproblem}\begin{rm}}{\end{rm}\end{itproblem}}
    {\begin{italgorithm}\begin{rm}}{\end{rm}\end{italgorithm}}

\newcommand{\Xd}{(X, d)}

\newcommand{\M}{\mathcal{M}}

\newcommand{\Mzero}{\M_0}
\newcommand{\Mone}{\M_1}
\newcommand{\Moneplus}{\M_1^+}
\newcommand{\Mplus}{\M^+}

\newcommand{\MX}{\M(X)}
\newcommand{\MzeroX}{\M_0(X)}
\newcommand{\MoneX}{\M_1(X)}

\newcommand{\Ezero}{E_0}

\newcommand{\Imu}{I(\mu)}

\newcommand{\Imunu}{I(\mu, \nu)}

\newcommand{\N}{\mathbb{N}}
\newcommand{\R}{\mathbb{R}}

\newcommand{\p}{\mathbb{P}}

\newcommand{\mbar}{M}

\newcommand{\ip}[2]{( #1 \mid #2 )}

\newcommand{\Bigip}[2]{\Bigl( #1 \Bigm| #2 \Bigr)}
\newcommand{\biggip}[2]{\biggl( #1 \biggm| #2 \biggr)}

\newcommand{\wstar}{\mbox{weak-$*$}}

\newcommand{\ts}{\textstyle}

\allowdisplaybreaks

\newcommand{\nwrefA}[1]{\ref{#1} of~\cite{NW1}}
\newcommand{\nwrefB}[1]{\ref{#1} of~\cite{NW2}}

\begin{document}

\title
{Distance Geometry in Quasihypermetric Spaces.~III}

\author{Peter Nickolas}
\address{School of Mathematics and Applied Statistics,
University of Wollongong, Wollongong, NSW 2522, Australia}
\email{peter\_\hspace{0.8pt}nickolas@uow.edu.au}
\author{Reinhard Wolf}
\address{Institut f\"ur Mathematik, Universit\"at Salz\-burg,
Hellbrunnerstrasse~34, A-5020 Salz\-burg, Austria}
\email{Reinhard.Wolf@sbg.ac.at}



\date{}

\begin{abstract}
Let $\Xd$ be a compact metric space and let
$\MX$ denote the space of all finite signed Borel measures on~$X$.
Define $I \colon \MX \to \R$ by
\[
\Imu = \int_X \! \int_X d(x,y) \, d\mu(x) d\mu(y),
\]
and set
$
\mbar(X) = \sup \Imu,
$
where $\mu$ ranges over
the collection of signed measures in $\MX$ of total mass~$1$.
This paper, with two earlier papers [Peter Nickolas and Reinhard Wolf,
\emph{Distance geometry in quasihypermetric spaces.\ I} and~\emph{II}],
investigates the geometric constant~$\mbar(X)$
and its relationship to the metric properties of~$X$ and the
functional-analytic properties of a certain subspace of $\MX$
when equipped with a natural semi-inner product.
Specifically, this paper explores links between the properties
of~$\mbar(X)$ and metric embeddings of~$X$, and the properties
of~$\mbar(X)$ when $X$ is a finite metric space.
\end{abstract}

\maketitle

\section{Introduction}
\label{Introduction3}

Let $\Xd$ be a compact metric space and let
$\MX$ denote the space of all finite signed Borel measures on~$X$.
Let $I \colon \MX \to \R$ be defined by
\[
\Imu = \int_X \! \int_X d(x,y) \, d\mu(x) d\mu(y),
\]
and set
\[
\mbar(X) = \sup \Imu,
\]
where $\mu$ ranges over $\MoneX$,
the collection of signed measures in $\MX$ of total mass~$1$.

Our interest in this paper and its predecessors \cite{NW1}
and~\cite{NW2} is in the properties
of the geometric constant $\mbar(X)$.
In~\cite{NW1}, we observed that if $\Xd$ does not have
the quasihypermetric property, then $\mbar(X)$ is infinite,
and thus the context of our study for the most part
is that of quasihypermetric spaces.
Recall (see \cite{NW1}) that $\Xd$ is quasihypermetric
if for all $n \in \N$, all $\alpha_1, \ldots, \alpha_n \in \R$
satisfying
$\sum_{i=1}^n \alpha_i = 0$, and all $x_1, \ldots, x_n \in X$,
we have
\[
\sum_{i,j=1}^n \alpha_i \alpha_j d(x_i, x_j) \leq 0.
\]

In the presence of the quasihypermetric property,
a natural semi-inner product space structure becomes available
on $\MzeroX$, the subspace of $\MX$ consisting of
all signed measures of mass~$0$. Specifically,
for $\mu, \nu \in \MzeroX$, we define
\[
(\mu \mid \nu) = -\Imunu,
\]
and denote the resulting semi-inner product space by $\Ezero(X)$.
The associated seminorm $\|\cdot\|$ on $\Ezero(X)$
is then given by
\[
\| \mu \| = \bigl[ -\Imu \bigr]^\frac{1}{2}.
\]

The semi-inner product space $\Ezero(X)$ is in many ways the key
to our analysis of the constant $\mbar(X)$.
In~\cite{NW1}, we developed the properties of
$\Ezero(X)$ in a detailed way, exploring in particular
the properties of several operators and functionals associated
with $\Ezero(X)$, some questions related to its topology,
and the question of completeness. Questions directly relating
to the constant $\mbar(X)$ were only examined in~\cite{NW1}
when they had a direct bearing on this general analysis.
In~\cite{NW2}, we discussed maximal measures (measures which attain
the supremum defining $\mbar(X)$), sequences of measures
which approximate the supremum when no maximal measure exists
and conditions implying or equivalent to the finiteness of $\mbar(X)$.

In this paper, building on the above work, we discuss
\begin{enumerate}
\item[(1)]
metric embeddings of $X$, both of a explicitly geometric type
and of a more abstract functional-analytic type, and
\item[(2)]
the properties of $\mbar(X)$ when $X$ is a finite metric space.
\end{enumerate}

We assume here that the reader has read \cite{NW1} and~\cite{NW2},
and we repeat their definitions and results here only as necessary.

\section{Definitions and Notation}
\label{chapter2b}

Let $\Xd$ (abbreviated when possible to~$X$) be a compact metric space. 
The diameter of $X$ is denoted by $D(X)$.
We denote by $C(X)$ the Banach space
of all real-valued continuous functions
on~$X$ equipped with the usual sup-norm.
Further,
\begin{itemize}
\item
$\MX$ denotes the space of all finite signed Borel measures on~$X$,
\item
$\Mzero(X)$ denotes the subspace of $\MX$
consisting of all measures of total mass~$0$,
\item
$\Mone(X)$ denotes the affine subspace of $\MX$
consisting of all measures of total mass~$1$,
\item
$\Mplus(X)$ denotes the set of all positive measures in~$\MX$, and
\item
$\Moneplus(X)$ denotes the intersection of $\Mplus(X)$ and $\Mone(X)$,
the set of all probability measures on~$X$.
\end{itemize}
For $x \in X$, the atomic measure at~$x$ is denoted by~$\delta_x$.

The following two functionals on measures play a central role in our work.
If $\Xd$ is a compact metric space, then for $\mu, \nu \in \MX$, we set
\[
\Imunu = \int_X \! \int_X d(x, y) \, d\mu(x) d\nu(y),
\]
and then
\[
\Imu = I(\mu, \mu).
\]
Also, a linear functional $J(\mu)$ on $\MX$ is
defined for each $\mu \in \MX$ by $J(\mu)(\nu) = \Imunu$
for all $\nu \in \MX$.
For $\mu \in \MX$, the function $d_\mu \in C(X)$ is defined by
\[
d_\mu(x) = \int_X d(x, y) \, d\mu(y)
\]
for $x \in X$.

For the compact metric space $\Xd$, we define
\[
\mbar(X) = \sup \bigl\{ \Imu: \mu \in \Mone(X) \bigr\}.
\]

\section{Metric Embeddings of Finite Spaces}
\label{embeddingsI}

Metric embeddings of various types have played a significant role
in work on the geometric properties of metric spaces.
In section~\ref{sect:qhm} of~\cite{NW1},
for example, we discussed briefly some connections between
the quasihypermetric property and $L^1$-embeddability and between
the quasihypermetric property and the metric embedding ideas of
Schoenberg~\cite{Sch3}. Also, embedding arguments based around and
extending Schoenberg's ideas were used in~\cite{AandS} by
Alexander and Stolarsky to obtain information on $M(X)$
when $X$ is a subset of euclidean space, and in~\cite{Ass1}  
by Assouad to characterize the hypermetric property
in finite metric spaces (see section~\ref{finitespaces} below
for the definition of the hypermetric property). 

In this and the following section, we apply metric embedding arguments
to the analysis of the constant~$\mbar$. In this section,
our arguments are for finite spaces,
and are of a more or less explicitly geometric character,
while in the following section, we use embedding arguments of
a functional-analytic character, and the results are for the case
of a general (usually compact) metric space.

As mentioned in section~\nwrefA{sect:qhm}, Schoenberg~\cite{Sch3} proved
that a separable metric space $\Xd$ is quasihypermetric if and only if
the metric space $(X, d^\frac{1}{2})$ can be embedded isometrically
in the Hilbert space~$\ell^2$. In particular, if $X$ is a finite space,
then $\Xd$ is quasihypermetric if and only if $(X, d^\frac{1}{2})$
can be embedded isometrically in a euclidean space
of suitable dimension. We will refer to an embedding
of $(X, d^\frac{1}{2})$ in a euclidean space or in Hilbert space
as a \textit{Schoenberg-embedding} or, for short,
an \textit{S-embedding} of~$X$.

Our results in this section relate the metric properties
of a space~$X$ which are our main interest
to the geometric properties of the S-embeddings of~$X$
and to the existence of invariant measures on~$X$
(see section~\nwrefB{maxinvmeasures}) of total mass~$1$.

First we have the following result, for the proof of which we make use
of some ideas developed by Assouad~\cite{Ass1}.

\begin{theorem}
\label{sphere100}
Let $\Xd$ be a finite metric space. If $\mbar(X) < \infty$,
then every S-embedding of~$X$ in a euclidean space lies on
\textup{(}the surface of\textup{)} a sphere. 
\end{theorem}

\begin{proof}
Suppose that $X = \{x_1, \ldots, x_n\}$ and that the S-embedding of~$X$
into the euclidean space $E = \R^m$ maps $x_i$ to $y_i \in E$
for $i = 1, \ldots, n$. We are seeking $z \in E$ such that
$\|y_i - z\|^2 = \|y_j - z\|^2$ for all $i, j$, and it is easy to see
that this relation holds for $z \in E$ if and only if
$\|y_1\|^2 - \|y_i\|^2 = 2 \ip{y_1-y_i}{z}$ for all~$i$.
Further, if we let $T$ denote the hyperplane
$\{(t_1, \ldots, t_n) : \sum_{i=1}^n t_i = 0\}$
in~$\R^n$, we see that the last relation holds if and only if
\[
\sum_{i=1}^n t_i \|y_i\|^2 = 2 \biggip{\sum_{i=1}^n t_i y_i}{z}
\]
for all $(t_i) \in T$. Let $\{e_k : k = 1, \ldots, m\}$
be an orthonormal basis for~$E$,
and define functionals $u$ and~$v_k$ for $k = 1, \ldots, m$
on the hyperplane~$T$ by setting
\[
u(t) = \sum_{i=1}^n t_i \|y_i\|^2
\quad\mbox{and}\quad
v_k(t) = \biggip{\sum_{i=1}^n t_i y_i}{e_k}
\]
for $t \equiv (t_i) \in T$. Then it is clear that there exists
$z \in E$ for which the condition above
holds if and only if there exist scalars $\{\alpha_k : k = 1, \ldots, m\}$
such that
\[
u(t) = \sum_{k=1}^m \alpha_k v_k(t)
\]
for all $t \in T$. We claim that this holds if and only if
\[
\bigcap_{k=1}^m \ker v_k \subseteq \ker u.
\]

To see this, suppose first that $\bigcap \ker v_k \subseteq \ker u$.
Now there exist $r \equiv (r_i)$ and $s_k \equiv (s_k^{(i)})$ in~$T$,
for $k = 1, \ldots, m$, such that
\[
u(t) = \ip{r}{t} \quad\mbox{and}\quad v_k(t) = \ip{s_k}{t}
\]
for all $t \in T$. Hence $\bigcap \ker v_k$ is the orthogonal
complement within~$T$ of the subspace of~$T$ generated by the~$\{s_k\}$,
and it follows that $r$ lies in the subspace generated by the~$\{s_k\}$.
Thus $u(t) = \sum_{k=1}^m \alpha_k v_k(t)$ for all $t \in T$,
for suitable scalars~$\{\alpha_k\}$.
The converse is clear, and so the claim holds.

Suppose that $s \equiv (s_i)$ satisfies $\sum s_i = 1$.
Straightforward manipulations then show that
\[
u(s) = \frac{1}{2} \sum_{i,j=1}^n s_i s_j \|y_i - y_j\|^2
       + \sum_{i,j=1}^n s_i s_j \ip{y_i}{y_j}
\]
and
\[
\sum_{k=1}^m v_k(s)^2 = \sum_{i,j=1}^n s_i s_j \ip{y_i}{y_j},
\]
giving
\[
u(s) = \frac{1}{2} \sum_{i,j=1}^n s_i s_j \|y_i - y_j\|^2 + \sum_{k=1}^m v_k(s)^2.
\]
(Note that $s$ is not in the domain~$T$ of the functionals $u$ and~$v_k$
as defined earlier, but we use the same symbols to denote the functions
whose values on~$s$ are defined by the same expressions.)

Given $t \equiv (t_i) \in T$, define $s \equiv (s_i) \in \R^n$
by setting $s_1 = t_1 + 1$ and $s_i = t_i$ for $i = 2, \ldots, n$,
so that $\sum s_i = 1$. Then we clearly have
\[
u(t) = u(s) - \|y_1\|^2
\quad\mbox{and}\quad
v_k(t) = v_k(s) - \ip{y_1}{e_k}
\]
for each~$k$.
Hence if $t \in \bigcap \ker v_k$, we have
\begin{eqnarray*}
u(t)
 & = & u(s) - \|y_1\|^2 \\
 & = & \frac{1}{2} \sum_{i,j=1}^n s_i s_j \|y_i - y_j\|^2
       + \sum_{k=1}^m v_k(s)^2 - \|y_1\|^2 \\
 & = & \frac{1}{2} \sum_{i,j=1}^n s_i s_j d(x_i, x_j)
       + \sum_{k=1}^m \bigl( v_k(t) + \ip{y_1}{e_k} \bigr)^2 
       - \|y_1\|^2 \\
 & = & \frac{1}{2} \sum_{i,j=1}^n s_i s_j d(x_i, x_j)
       + \sum_{k=1}^m \ip{y_1}{e_k}^2 
       - \|y_1\|^2 \\
 & = & \frac{1}{2} \sum_{i,j=1}^n s_i s_j d(x_i, x_j) \\
 & \leq & \frac{1}{2} \mbar(X).
\end{eqnarray*}
But since this holds for all $t \in \bigcap \ker v_k$ and $\mbar(X)$
is finite, we must have $u(t) = 0$ for all $t \in \bigcap \ker v_k$.
Thus $\bigcap \ker v_k \subseteq \ker u$, and the result follows.
\end{proof}

We show later (Theorem~\ref{sphere500}) that the above implication holds
when $X$ is a general compact metric space, with the corresponding sphere
then lying in general in the Hilbert space~$\ell^2$.

In \cite{AandS}, Alexander and Stolarsky made use of S-embeddings
on spheres to derive interesting results on~$\mbar$ and related matters
for subsets of euclidean spaces.
In the following result, we gather together some of their observations,
specialized to the case of finite spaces, but generalized
to the non-euclidean case, along with some new observations.

Recall (see~\cite{NW1}) that for a compact metric space $\Xd$,
we write $M^+(X) = \sup \{ \Imu : \mu \in \Moneplus(X) \}$.

\begin{theorem}
\label{sphere200}
Let $(X = \{x_1, \ldots, x_n\}, d)$ be a finite metric space,
and suppose that $X$ is S-embedded as the set\/ $Y = \{y_1, \ldots, y_n\}$
on a sphere~$S$ of radius~$r$ in some euclidean space,
where the S-embedding maps $x_i$ to~$y_i$ for $i = 1, \ldots, n$.
Then we have the following.
\begin{enumerate}
\item
$\mbar(X) \leq 2r^2$.
\item
There exists a maximal measure on~$X$.
\end{enumerate}
If further the S-embedding of~$X$ is into a euclidean space of minimal dimension,
then we have the following.
\begin{enumerate}
\setcounter{enumi}{2}
\item
$\mbar(X) = 2r^2$.
\item
$M^+(X) = 2(r^2 - s^2)$, where $s$ is the distance from the centre of~$S$
to the convex hull of\/~$Y$.
\item
If $w_1, \ldots, w_n \in \R$ are such that
$\sum_{i=1}^n w_i = 1$, then $\sum_{i=1}^n w_i \delta_{x_i}$ is a maximal measure
on~$X$ if and only if $\sum_{i=1}^n w_i y_i$ is the centre of~$S$.
\end{enumerate}
\end{theorem}

\begin{proof}
Suppose without loss of generality that the centre of the
sphere~$S$ is~$0$. If $w_1, \ldots, w_n \in \R$
satisfy $\sum_{i=1}^n w_i = 1$, then a straightforward calculation
(cf.~Lemma~3.2 of~\cite{AandS}) gives
\[
\sum_{i,j=1}^n w_i w_j d(x_i, x_j)
 = \sum_{i,j=1}^n w_i w_j \|y_i - y_j\|^2
 = 2r^2 - 2 \biggl\| \sum_{i=1}^n w_i y_i \biggr\|^2,
\]
and it follows that
\[
\mbar(X)
 =
2r^2\
 - 2 \inf
       \biggl\{
           \Bigl\| \sum_{i=1}^n w_i y_i \Bigr\|^2 : \sum_{i=1}^n w_i = 1
       \biggr\}
\]
and that
\[
M^+(X)
 =
2r^2
 - 2 \inf
       \biggl\{ \Bigl\| \sum_{i=1}^n w_i y_i \Bigr\|^2 :
         w_1, \ldots, w_n \geq 0 \mbox{ and } \sum_{i=1}^n w_i = 1
       \biggr\}.
\]
This gives (1), and then (2) follows by Theorem~\nwrefB{theoremDcorollary}.
Now assume that the S-embedding of~$X$ is into $\R^k$,
where $k$ is the minimum dimension possible,
so that the affine hull of~$Y$ is~$\R^k$.
Then there exist $w_1, \ldots, w_n$ with $\sum_{i=1}^n w_i = 1$
such that $\sum_{i=1}^n w_i y_i = 0$, and it follows from
the expression derived above for $\mbar(X)$
that $\mbar(X) = 2r^2$, and we have~(3).
The expression for $\mbar(X)$ also clearly gives~(5).
Finally, since the distance~$s$ from the centre of~$S$
to the convex hull of~$Y$ is
\[
\inf \biggl\{ \Bigl\| \sum_{i=1}^n w_i y_i \Bigr\| :
    w_1, \ldots, w_n \geq 0 \mbox{ and } \sum_{i=1}^n w_i = 1 \biggr\},
\]
the expression derived above for $M^+(X)$ gives~(4).
\end{proof}

\begin{corollary}
\label{sphere300}
In the circumstances of the theorem,
\begin{enumerate}
\item
there is a unique maximal measure on~$X$ if and only if
the S-embedded set\/~$Y$ is affinely independent, and
\item
if the S-embedding is into a space of minimal dimension,
then the maximal measure on~$X$ given by the theorem
is a probability measure if and only if
the centre of the sphere~$S$ is in the convex hull of\/~$Y$.
\end{enumerate}
\end{corollary}

\begin{proof}
Suppose that the S-embedding is into a space of minimal dimension.
Then by part~(5) of the theorem, there is a unique maximal measure on~$X$
if and only if $0$ can be written as an affine combination
of $y_1 \ldots, y_n$ in a unique way, and this is the case
if and only if $Y$ is a maximal affinely independent set.
By the argument used for part~(3) of the theorem, this is equivalent
in the general case to the affine independence of~$Y$, giving~(1).
Assertion~(2) is immediate from part~(5) of the theorem.
\end{proof}

Now we can prove the result alluded to earlier which expresses
metric properties of~$X$ as equivalent geometric conditions
on S-embeddings of~$X$ and also as equivalent conditions
on $d$-invariant measures of mass~$1$ on~$X$.
(We prefer to speak of invariant measures of mass~$1$ here rather than
of maximal measures---%
see section~\nwrefB{maxinvmeasures} for the relevant definitions---%
but recall that by Theorem~\nwrefB{2.2}
these classes of measures coincide in any compact quasihypermetric space.)

\begin{theorem}
\label{sphere400}
Let $\Xd$ be a finite quasihypermetric space.
\begin{enumerate}

\item
The following conditions are equivalent.
\begin{enumerate}
\item
$\mbar(X) < \infty$.
\item
There exists a $d$-invariant measure in $\Mone(X)$. 
\item
Some S-embedding of~$X$ in a euclidean space lies on a sphere.
\item
Every S-embedding of $X$ in a euclidean space lies on a sphere.
\end{enumerate}

\item
The following conditions are equivalent.
\begin{enumerate}
\item
$M^+(X) = \mbar(X)$.
\item
There exists a $d$-invariant measure in $\Moneplus(X)$.
\item
Some S-embedding of~$X$ in a euclidean space of minimal dimension
lies on a sphere whose centre is in the convex hull of the S-embedded set.
\item
Every S-embedding of $X$ in a euclidean space of minimal dimension
lies on a sphere whose centre is in the convex hull of the S-embedded set.
\end{enumerate}

\item
The following conditions are equivalent.
\begin{enumerate}
\item
$X$ is strictly quasihypermetric.
\item
There exists a unique $d$-invariant measure in $\Mone(X)$.
\item
Some S-embedding of~$X$ in a euclidean space
is an affinely independent set.
\item
Every S-embedding of $X$ in a euclidean space
is an affinely independent set.
\end{enumerate}

\end{enumerate}
\end{theorem}

\begin{proof}
(1) Theorem~\nwrefB{2.2} shows that (b) implies~(a),
Theorem~\ref{sphere100} (of the present paper) shows that (a) implies~(d),
the result of Schoenberg~\cite{Sch3} quoted before Theorem~\ref{sphere100}
shows that there exists an S-embedding of~$X$ into a euclidean space,
from which it follows that (d) implies~(c),
and Theorem~\ref{sphere200}, with Theorem~\nwrefB{2.2},
shows that (c) implies~(b).

\smallskip
(2) Assume~(a), and consider any S-embedding of~$X$
on a sphere in a euclidean space of minimal dimension.
Then using~(a) and parts (3) and~(4) of Theorem~\ref{sphere200},
we find that the distance of the centre of the sphere
from the convex hull of the embedded set is~$0$, and compactness yields~(d).
That (d) implies~(c) is shown as in part~(1), and
Corollary~\ref{sphere300}, with Theorem~\nwrefB{2.2},
shows that (c) implies~(b).
Assume~(b) and let $\mu \in \Moneplus(X)$ be $d$-invariant.
Then Theorem~\nwrefB{2.2} shows that $\mu$ has value $\mbar(X)$.
But since $\mu \in \Moneplus(X)$, it follows that
$M^+(X) = \mbar(X)$, and we have~(a).

\smallskip
(3) Corollary~\ref{sphere300} and Theorem~\nwrefB{2.2}
show that (c) implies~(b), those results together with part~(1)
show that (b) implies~(d),
and the fact that (d) implies~(c) is shown as earlier.
Assume that $X$ is strictly quasihypermetric. Then by Theorem~\nwrefB{2.13}
(see Theorem~\ref{2.13new} below),
we have $\mbar(X) < \infty$, so by part~(1), there exists a $d$-invariant
$\mu \in \Mone(X)$, which is unique by part~(4) of Theorem~\nwrefB{2.2}.
Thus, (a) implies~(b).
Now assume that $X$ is not strictly quasihypermetric. If $\mbar(X) = \infty$,
then by Theorem~\nwrefB{2.2}, $X$ has no $d$-invariant measure of mass~$1$.
If $\mbar(X) < \infty$, then by part~(1), there exists a
$d$-invariant $\mu \in \Mone(X)$.
Since $X$ is not strictly quasihypermetric, it follows from parts (2) and~(5)
of Lemma~\nwrefA{2.7} that there exists a non-zero $d$-invariant measure
$\nu \in \Mzero(X)$ (which, by Theorem~\nwrefB{2.13}, has value~$0$).
It follows that $\mu + \nu \in \Mone(X)$, that $\mu + \nu$ is
$d$-invariant, and that $\mu + \nu \neq \mu$, so that
there is more than one $d$-invariant measure of mass~$1$ on~$X$.
Thus, (b) implies~(a), completing the proof.
\end{proof}

\begin{remark}
\label{Assouad}
In~\cite{Ass1}, Assouad develops characterizations
of the hypermetric property and the property of $L^1$-embeddability
of a finite metric space. A space has one of these properties
if it can be S-embedded on a sphere in euclidean space
in such a way as to satisfy an additional lattice-theoretical constraint,
stronger in the second case than the first, since $L^1$-embeddability
implies the hypermetric property (cf.~our Theorem~\ref{2.15} below,
the proof of which can be adapted routinely to show this).
It follows by Theorem~\ref{sphere200} that such spaces
have $\mbar$ finite (cf.~Theorem~\ref{2.16}).
It would be interesting to know if there are characterizations
of these two properties in terms of invariant measures.
\end{remark}

\section{Metric Embeddings of General Spaces}
\label{embeddingsII}

We begin by noting the following result,
which relates the value of~$\mbar$ on a general compact metric space~$X$
and the value of~$\mbar$ on the finite subsets of~$X$
(see also Theorem~\ref{2.28} below).  
The result can easily be proved either by an argument
similar to that needed to show that (1) implies~(3)
in Theorem~\nwrefA{qhmconds},
or by adapting the proof of Lemma~3.3 of~\cite{AandS}.

\begin{theorem}
\label{sphere50}
If $X$ is a compact metric space, then $\mbar(X)$ is the supremum
of the values $\mbar(F)$ for finite subsets~$F$ of~$X$.
\end{theorem}

\begin{definition}
\label{2.14}
Let $\Xd$ be a metric space. We say that
\textit{$X$ admits an $L^1$-embedding} if there exists a probability
space $(\Omega, A, \p)$ and a mapping $i \colon X \to L^1(\Omega)$
such that $d(x, y) = \|i(x) - i(y) \|$ for all $x, y \in X$.
Further, we say that this embedding is \textit{uniformly bounded}
if $\sup|i(x)(\omega)| < \infty$, where $x$ and~$\omega$ range over $X$ and~$\Omega$,
respectively.
\end{definition}

Assertion (1) of the following result is well known,
as is the stronger assertion that an $L^1$-embeddable space is hypermetric
(for the definition of the hypermetric property,
see section~\ref{finitespaces} below).

\begin{theorem}
\label{2.15}
Let $X$ be a metric space admitting an $L^1$-embedding. Then we have the following.
\begin{enumerate}
\item[(1)]
$X$ is quasihypermetric.
\item[(2)]
If additionally $X$ is compact and the given $L^1$-embedding
is uniformly bounded, with $|i(x)(\omega)| \leq K$
for some $K \geq 0$ and for all $x \in X$ and $\omega \in \Omega$,
then $\mbar(X) \leq K$.
\end{enumerate}
\end{theorem}

\begin{proof}
(1)
Consider $n \in \N$, $x_1, \ldots, x_n \in X$ and $\alpha_1, \ldots, \alpha_n \in \R$
such that $\alpha_1 + \cdots + \alpha_n = 0$. Then, since $\R$ is quasihypermetric, we have
\begin{eqnarray*}
\sum_{i, j = 1}^n \alpha_i \alpha_j d(x_i, x_j)
 & = &
\sum_{i, j = 1}^n \alpha_i \alpha_j \bigl\| i(x_i)-i(x_j) \bigr\| \\
 & = &
\int_\Omega
  \Bigl(
    \sum_{i,j=1}^n \alpha_i \alpha_j
        \bigl| i(x_i)(\omega) - i(x_j)(\omega) \bigr|
  \Bigr)
d\kern0.5pt\p\omega \\
 & \leq &
0.
\end{eqnarray*}
Therefore, $X$ is quasihypermetric.

\smallskip
(2)
Consider $n \in \N$, $x_1, \ldots, x_n \in X$ and $\alpha_1, \ldots, \alpha_n \in \R$
such that $\alpha_1 + \cdots + \alpha_n = 1$. As before, we have
\[
\sum_{i,j=1}^n \alpha_i \alpha_j d(x_i, x_j)
 =
\int_\Omega
  \Bigl(
    \sum_{i,j=1}^n
        \alpha_i \alpha_j \bigl| i(x_i)(\omega) - i(x_j)(\omega) \bigr|
  \Bigr)
d\kern0.5pt\p\omega.
\]
Let $K = \sup \{ |i(x)(\omega)| : x \in X, \omega \in \Omega \}$.
By assumption, $K < \infty$. Applying Corollary~\nwrefB{2.3} to the
interval $[-K, K]$ gives
$\sum_{i,j=1}^n \alpha_i \alpha_j | i(x_i)(\omega) - i(x_j)(\omega)| \leq K$
for all $\omega \in \Omega$,
and so $\sum_{i,j=1}^n \alpha_i \alpha_j d(x_i, x_j) \leq K$.
Therefore, we have $\mbar(F) \leq K$ for all finite subsets~$F$ of~$X$,
and it follows by Theorem~\ref{sphere50} that $\mbar(X) \leq K$.
\end{proof}

We say that a real normed linear space $(E, \|\cdot\|)$
is quasihypermetric if the corresponding metric space $(E, d)$
is quasihypermetric, where $d$ is the norm-induced metric on~$E$.

We wish next to discuss some properties of subsets
of finite-dimensional real normed linear spaces.
Of course, every such space is isometrically isomorphic
to a space $(\R^n, \|\cdot\|)$ for some~$n$ and some norm $\|\cdot\|$,
and so it suffices to restrict attention to subsets of spaces of this type.
We recall the well known fact that for any fixed space $(\R^n, \|\cdot\|)$,
the following three conditions are equivalent.
\begin{enumerate}
\item[(1)]
The space $(\R^n, \|\cdot\|)$ is quasihypermetric.
\item[(2)]
The space $(\R^n, \|\cdot\|)$ is isometrically isomorphic
to a subspace of $L^1([0, 1])$ (the space is $L^1$-embeddable).
\item[(3)]
The norm $\|\cdot\|$ admits a so-called \textit{L\'{e}vy representation};
that is, there exist $\alpha > 0$ and a probability measure~$\p$
on the euclidean unit sphere~$S^{n-1}$ in~$\R^n$ such that
\[
\|x\| = \alpha \int_{S^{n-1}} \bigl| (x \mid \omega) \bigr|
        \, d\kern0.5pt\p (\omega)
\]
for all $x \in \R^n$. 
\end{enumerate}
(For a proof, one can combine Corollaries 1.1 and~1.3
of~\cite{Wit} with Corollaries 2.6 and~6.2 of~\cite{Bol}.)

We have seen that $\mbar(X)$ may be infinite when $X$ is
a compact (or even finite) quasihypermetric space.
In the presence of a linear structure,
however, we have the following result, which generalizes the
euclidean case proved in Theorem~3.8 of~\cite{AandS}.

\begin{theorem}
\label{2.16}
Suppose that $(\R^n, \|\cdot\|)$ is quasihypermetric,
and let $X$ be a subset of\/~$\R^n$ which is compact
when equipped with the norm-induced metric. Then
\begin{enumerate}
\item[(1)]
$\mbar(X) < \infty$ and
\item[(2)]
there exists $c > 0$ such that
$|I(\mu_1) - I(\mu_2)| \leq c \| \mu_1 - \mu_2 \|$
for all $\mu_1, \mu_2 \in \Moneplus(X)$.
\end{enumerate}
\end{theorem}

\begin{proof}
Using the comments above, choose $\alpha > 0$
and a probability measure~$\p$ on the euclidean unit sphere~$S^{n-1}$
in~$\R^n$ such that
\[
\|x\| = \alpha \int_{S^{n-1}} \bigl| (x \mid \omega) \bigr|
        \, d\kern0.5pt\p (\omega)
\]
for all $x \in \R^n$. Define $i \colon X \to L^1(S^{n-1}, \p)$ by setting
$i(x)(\omega) = \alpha(x \mid \omega)$ for $x \in X$ and $\omega \in S^{n-1}$.
Then $\|x - y\| = \| i(x) - i(y) \|$ for all $x, y \in X$, and,
by the compactness of~$X$, there exists~$K$ such that
$|i(x)(\omega)|
 = \alpha \, |(x \mid \omega)|
 \leq \alpha \, \| x \|_2
 \leq K$
for all $x \in X$ and $\omega \in S^{n-1}$,
where $\|\cdot\|_2$ denotes the euclidean norm on~$\R^n$.
Thus $X$ admits a uniformly bounded $L^1$-embedding,
and now Theorem~\ref{2.15} above and part~(6) of Theorem~\nwrefA{2.10}
(see also Remark~\nwrefA{2.12}) complete the proof.
\end{proof}

\begin{remark}
\label{2.16remark}
It is well known that any finite metric space can be
isometrically embedded in~$\R^n$ with the $\infty$-norm for suitable~$n$
(see, for example, part~(1) of Lemma~\nwrefA{newlemma}),
and it is also well known that this normed space
is non-quasihypermetric if $n \geq 3$
(see section~\nwrefA{sect:qhm}). Thus it is the quasihypermetric property
of the enclosing normed space rather than of the embedded metric space
that is crucial for the conclusions of the theorem. 
\end{remark}

\begin{theorem}
\label{theoremE}
Let $(X, d)$ be a compact metric space with $\mbar(X) < \infty$.
Then there exists a mapping $i$ of $X$ into separable Hilbert space
such that
\begin{enumerate}
\addtolength{\itemsep}{1mm}
\item
$\| i(x) \| = \bigl( \frac{1}{2} \mbar(X) \bigr)^{\frac{1}{2}}$
for all $x \in X$ and
\item
$\| i(x) - i(y) \| ^2 = d(x, y)$ for all $x, y \in X$.
\end{enumerate}
\end{theorem}

\begin{proof}
We remark first that $X$ is quasihypermetric, by Theorem~\nwrefA{2.6}.
Define the semi-inner product space $Y$ by setting $Y = \MX$ and
$(\mu \mid \nu ) := \mbar(X) \mu (X) \nu (X) - I(\mu, \nu)$
for $\mu, \nu \in \MX$.
Let $Y_0 = \{ \mu \in Y : \mbox{ } \| \mu \| =0 \}$.
Now let $H$ be the completion of the inner product space $Y / Y_0$,
and define $i \colon X \to H$ by
\[
i(x) := \frac{1}{\sqrt{2}} \delta_x + Y_0
\]
for $x \in X$. Now
\begin{eqnarray*}
\bigl\| i(x) \bigr\| ^2
 & = &
\Bigip%
  {\frac{1}{\sqrt{2}} \delta_x + Y_0}%
  {\frac{1}{\sqrt{2}} \delta_x + Y_0} \\[1mm]
 & = &
\ts\frac{1}{2} \| \delta_x \| ^2 \\[1mm]
 & = &
\ts\frac{1}{2} \mbar(X)
\end{eqnarray*}
for all $x \in X$, and
\begin{eqnarray*}
\bigl\| i(x) - i(y) \bigr\| ^2
 & = &
\Bigip%
  {\frac{1}{\sqrt{2}} (\delta_x - \delta_y) + Y_0}%
  {\frac{1}{\sqrt{2}} (\delta_x - \delta_y) + Y_0} \\[1mm]
 & = &
\ts\frac{1}{2} \| \delta_x - \delta_y \|^2 \\[1mm]
 & = &
-\ts\frac{1}{2} I (\delta_x - \delta_y, \delta_x - \delta_y) \\[1mm]
 & = &
d(x, y)
\end{eqnarray*}
for all $x, y \in X$.
Finally, the image $i(X)$ of~$X$ in~$H$ is homeomorphic to~$X$,
and therefore separable, and standard arguments show that
the closure of the subspace generated by $i(X)$ is separable.
\end{proof}

To continue our discussion of metric embeddings,
we require the following result, which gives
more detailed information than Theorem~\ref{sphere50} above
on the relation between~$\mbar(X)$ and the value
of~$\mbar$ on the finite subsets of~$X$.  

\begin{theorem}
\label{2.28}
Let $\Xd$ be a compact quasihypermetric space.
Let $(x_n)_{n\geq1}$ be any dense sequence in~$X$
and write $X_n = \{x_1, \ldots , x_n \}$ for each $n \in \N$.
Then $\mbar(X_n) \uparrow \mbar(X)$ as $n \to \infty$.
\end{theorem}

\begin{proof}
The values $\mbar(X_n)$ are obviously non-decreasing,
so that convergence of $\mbar(X_n)$ to $\mbar(X)$ is all
that we need to prove.
Suppose first that $\mbar(X) < \infty$,
so that also $\mbar(X_n) < \infty$ for all~$n$.
Applying Theorems \ref{theoremDcorollary} and~\ref{2.2} of~\cite{NW2}
to~$X_n$ for each $n \geq 2$, we obtain a measure $\mu_n \in \Mone(X_n)$
such that $d_{\mu_n}(x_i) = \mbar(X_n)$
for all $i$ such that $1 \leq i \leq n$.

If $n$ and $m$ are integers with $n > m$, then we have
\begin{eqnarray*}
\| \mu_n - \mu_m\|^2
 & = &
2 I(\mu_n, \mu_m) - I (\mu_n)- I(\mu_m) \\
 & = &
2 \mu_m (d_{\mu_n})-\mu_n (d_{\mu_n})- \mu_m (d_{\mu_m}) \\
 & = &
2 \mbar (X_n)- \mbar (X_n)- \mbar(X_m) \\
 & = &
\mbar(X_n) - \mbar(X_m).
\end{eqnarray*}
Therefore, $\mbar(X_m) \leq \mbar(X_n) = I(\mu_n) \leq
\mbar(X) < \infty$ whenever $n >m$, and so there exists
$\beta \in \R$ such that $\mbar(X_n) \uparrow \beta$ as
$n \to \infty$. Hence $\| \mu_n - \mu_m \| \to 0$ as $n, m \to \infty$.
By part~(5) of Theorem~\nwrefA{2.10}
we conclude that $d_{\mu_n}$ is a Cauchy sequence in $C(X)$,
and hence that there exists $f \in C(X)$ such
that $d_{\mu_n} \to f$ in $C(X)$ as $n \to \infty$.

Now fix $k \geq 1$, and let $n \geq \max (k, 2)$.
Since $d_{\mu_n}(x_k) = \mbar (X_n)$, we have $d_{\mu_n}(x_k) \to \beta$
as $n \to \infty$, and hence $f(x_k) = \beta$ for all $k \geq 1$.
Since $x_n$ is a dense sequence in~$X$ and $f$ is continuous on~$X$,
we have $f(x) = \beta$ for all $x \in X$, and hence
$d_{\mu_n} \to \beta \cdot \underline{1}$ in $C(X)$.
Thus we have shown that $\mu_n$ is a $d$-invariant sequence
with value~$\beta$ and that $I(\mu_n) \uparrow \beta$.
An application of Theorem~\nwrefB{theoremD} now gives $\mbar(X) = \beta$,
as required.

Now suppose that $\mbar(X) = \infty$.
If $\mbar(X_{n_0}) = \infty$ for any~$n_0$,
then clearly $\mbar(X_n) = \infty$ for all $n \geq n_0$,
and there is nothing to prove, so suppose that
$\mbar(X_n) < \infty$ for all~$n$. Fix $K > 0$.
By Theorem~\ref{sphere50}, there is a finite subset
$Y = \{y_1, \ldots, y_m\}$ of~$X$ such that $\mbar(Y) > K$,
and hence a measure $\mu \in \Mone(Y)$ such that $I(\mu) > K$.
Write $\mu = \sum_{i=1}^m w_i \delta_{y_i}$ for suitable
$w_1, \ldots, w_m \in \R$.
For each~$i$, pick a sequence $x_{n_{i,k}}$ with members chosen
from the dense sequence $(x_n)_{n\geq1}$ such that
$x_{n_{i,k}} \to y_i$ as $k \to \infty$.
Then, setting $\mu_k = \sum_{i=1}^m w_i \delta_{x_{n_{i,k}}}$
for each~$k$, we clearly have $\mu_k \to \mu$ \wstar{} in $\Mone(X)$
as $k \to \infty$. Hence, by Theorem~\nwrefA{contoncompact}
(or Corollary~\nwrefA{cont2}), we have $I(\mu_k) \to I(\mu)$
as $k \to \infty$. It follows that for sufficiently large~$N$
there exists $\nu \in \Mone(X_N)$ such that $I(\nu) > K$.
Therefore, $\mbar(X_n) \to \infty$, as required.
\end{proof}

\begin{theorem}
\label{newsphere100}
Let $\Xd$ be a compact metric space with $\mbar(X) < \infty$.
Let $i \colon X \to H$ be an S-embedding of~$X$ into a Hilbert space~$H$.
Then $i(X)$ lies on a sphere in~$H$ of radius~$r$, where $\mbar(X) = 2r^2$.
\end{theorem}

\begin{proof}
As in Theorem~\ref{2.28}, choose a dense sequence $(x_n)_{n\geq1}$
in~$X$, write $X_n = \{x_1, \ldots , x_n \}$
for each $n \in \N$, and let $\mu_n \in \Mone(X_n)$ be a maximal,
and hence $d$-invariant, measure on~$X_n$.
We may assume that $x_i \neq x_j$ when $i \neq j$.
For any $\mu \in [\delta_{x_1}, \delta_{x_2}, \ldots]$,
the linear span of $\{ \delta_{x_1}, \delta_{x_2}, \ldots \}$,
we have $\mu = \sum_{k=1}^n \beta_k \delta_{x_k}$ for suitable $n \in \N$
and $\beta_1, \ldots, \beta_n \in \R$, and we
define $z_\mu \in H$ by $z_\mu = \sum_{k=1}^n \beta_k i(x_k)$.

Note that if $\mu(X) = 0$, then we have $\sum_k \beta_k = 0$,
from which it follows that
\begin{eqnarray*}
\| \mu \|^2
 & = & -I(\mu) \\[1mm]
 & = & -\sum_{k=1}^n \sum_{\ell=1}^n \beta_k \beta_\ell d(x_k, x_\ell) \\[1mm]
 & = & -\sum_{k=1}^n \sum_{\ell=1}^n
        \beta_k \beta_\ell \bigl\| i(x_k) - i(x_\ell) \bigr\|^2 \\[1mm]
 & = & 2 \| z_\mu \|^2,
\end{eqnarray*}
giving $\| z_\mu \|^2 = \frac{1}{2} \|\mu\|^2$.
Also note that if $\mu(X) = 1$, then another straightforward
calculation gives
$\| z_\mu - i(x) \|^2 = d_\mu(x) - \frac{1}{2} I(\mu)$ for all $x \in X$.

Now $(\mu_m - \mu_n)(X) = 0$ for all $m$ and~$n$,
so we can apply the first observation above, obtaining
$\| z_{\mu_m} - z_{\mu_n} \|^2 = \frac{1}{2} \| \mu_m - \mu_n \|^2$
for all $m$ and~$n$. Also, by the proof of Theorem~\ref{2.28},
the measures~$\mu_n$ form a $d$-invariant sequence in~$X$,
and so we have $\| z_{\mu_m} - z_{\mu_n} \| \to 0$
as $m, n \to \infty$. Hence, as $H$ is complete,
there exists $z \in H$ such that $z_{\mu_n} \to z$ as $n \to \infty$.
Since $\mu_n(X) = 1$ for each $n \in \N$,
we can apply the second observation above, obtaining
$\| z_{\mu_n} - i(x) \|^2 = d_{\mu_n}(x) - \frac{1}{2} \mbar(X_n)$
for all $n \in \N$ and $x \in X$. Finally, taking limits and using 
Theorem~\ref{2.28}, we have $\| z - i(x) \|^2 = \frac{1}{2} \mbar(X)$
for all $x \in X$, and the result follows.
\end{proof}

We conclude our discussion of embeddings by showing that the first part of
Theorem~\ref{sphere400} generalizes in a natural way to the general compact case.

\begin{theorem}
\label{sphere500}
Let $\Xd$ be a compact quasihypermetric space.
Then the following conditions are equivalent.
\begin{enumerate}
\item
$\mbar(X) < \infty$.
\item
There exists a $d$-invariant sequence in $\Mone(X)$.
\item
Some S-embedding of~$X$ in the Hilbert space~$\ell^2$ lies on a sphere. 
\item
Every S-embedding of~$X$ in the Hilbert space~$\ell^2$ lies on a sphere. 
\end{enumerate}
\end{theorem}

\begin{proof}
The equivalence of (1) and~(2) is given by Corollary~\nwrefB{theoremDnew},
the fact that (1) implies~(4) is given by Theorem~\ref{newsphere100},
and the fact that (4) implies~(3) is given by an application of Schoenberg's
result as in the proof of Theorem~\ref{sphere400}.

Suppose that $X$ can be S-embedded on a sphere
of radius~$r$ in~$\ell^2$. Clearly every finite subset~$F$
of~$X$ can then be S-embedded on a sphere of radius at most~$r$ in
a suitable euclidean space, and hence satisfies $\mbar(F) \leq 2r^2$,
by Theorem~\ref{sphere200}. It is now immediate by Theorem~\ref{sphere50}
that $\mbar(X) < \infty$. Thus, (3) implies~(1), completing the proof.
\end{proof}

\section{$\mbar(X)$ in Finite Spaces}
\label{finitespaces}

In this paper and the earlier paper~\cite{NW2}
we have derived several results about the constant~$\mbar(X)$
in a finite metric space~$X$, and have introduced a number of finite
metric spaces or classes of such spaces as examples and counterexamples.
The examples have typically been constructed so as to have the minimum
number of elements consistent with the phenomenon under discussion.

Our main general result on finite spaces in the present paper so far
has been Theorem~\ref{sphere400} above, and Theorem~\nwrefB{2.13} was
the main such result in the earlier paper. We reproduce the latter
result here for convenience.

\begin{theorem}[= Theorem~\nwrefB{2.13}]
\label{2.13new}
Let $\Xd$ be a finite quasihypermetric space. Then we have the following.
\begin{enumerate}
\item[(1)]
If $X$ is strictly quasihypermetric, then $\mbar(X) < \infty$.
\item[(2)]
If $X$ is not strictly quasihypermetric,
then $\mbar(X) < \infty$ if and only if there exists no
$d$-invariant measure $\mu \in \Mzero(X)$ with value $c \neq 0$.
\end{enumerate}
\end{theorem}

In this section, we develop further results about finite spaces,
and in particular settle some of the minimality questions
raised by our examples.

When the space $X$ is finite, the question of the finiteness of $\mbar(X)$
can be resolved by a straightforward algebraic test,
according to the next result, which also gives an algorithm
for the computation of $\mbar(X)$ when it is finite.
We note that Alexander and Stolarsky~\cite[Theorem~3.3]{AandS}
give a simple algorithm involving the solution of a system
of linear equations for the computation of $\mbar(X)$
when $X$ is a (strictly quasihypermetric) finite subset of euclidean space.

\begin{theorem}
\label{alg1}
Let $(X = \{x_1, \ldots, x_n\}, d)$ with $n \geq 2$
be a finite quasihypermetric space.
Consider the linear system $D w = \mathbf{1}$, where $D$ is the
distance matrix $\bigl( d(x_i, x_j) \bigr)_{i,j=1}^n$
and $\mathbf{1}$ is the vector $(1, \ldots, 1)^T$ of length~$n$.
Then a solution $w = (w_1, \ldots, w_n)^T$ to the system exists.
Further,
\begin{enumerate}
\setlength{\itemsep}{1mm}
\item[(1)]
if\/ $\sum_{i=1}^n w_i = 0$, then $\mbar(X) = \infty$, and
\item[(2)]
if\/ $\sum_{i=1}^n w_i = w_0 \neq 0$, then $\mbar(X) = 1/w_0 < \infty$,
and
\[
(1/w_0) \sum_{i=1}^n w_i \delta_{x_i} \in \Mone(X)
\]
is a maximal measure on~$X$.
\end{enumerate}
\end{theorem}

\begin{proof}
If $\mbar(X) = \infty$, part~(1) of Theorem~\ref{2.13new}
implies that $X$ is not strictly quasihypermetric,
and then part~(2) of Theorem~\ref{2.13new}
implies that there exists an invariant measure $\mu \in \Mzero(X)$
with some non-zero value~$c$.
We therefore clearly have a solution~$w$ to the linear system.
If $\mbar(X) < \infty$, then there exists an invariant measure
$\mu \in \Mone(X)$, by Theorem~\nwrefB{theoremDcorollary}.
This measure has value $\mbar(X)$ by Theorem~\nwrefB{2.2},
and since $n \geq 2$ we have $\mbar(X) > 0$.
We therefore again have a solution to the linear system.
Statement~(1) now follows from Theorem~\nwrefB{2.9.5new},
and statement~(2) from Theorem~\nwrefB{2.2}.
\end{proof}

\begin{remark}
\label{algremark1}
Implicit in the statement and proof of the last theorem is the fact
that $\sum_{i=1}^n w_i$ has the same value
for every solution~$w$ to the system $D w = \mathbf{1}$.
Also, since this system always has a solution,
the matrix~$D$ is non-singular if and only if the system
has exactly one solution. 
\end{remark}

\begin{remark}
\label{algremark2}
If $X$ in the theorem is strictly quasihypermetric,
then the distance matrix~$D$ is in fact non-singular.
Indeed, $D$ is the natural matrix representation
of the operator $T \colon \MX \to C(X)$ defined by $T(\mu) = d_\mu$
for $\mu\in \MX$, which is discussed and used extensively in~\cite{NW1}.
If $X$ is strictly quasihypermetric then Theorem~\nwrefA{2.19} shows that
$T$ is an injection (and hence, since $X$ is finite, a bijection),
and so $D$ is non-singular.

Moreover, Theorem~\nwrefA{2.20} implies that if $\mbar(X) < \infty$,
then $D$ is non-singular only if $X$ is strictly quasihypermetric.
Example~\nwrefB{4ptexample} provides an example
of a $4$-point space~$X$ which is quasihypermetric but not strictly quasihypermetric
and satisfies $\mbar(X) < \infty$,
and for which $D$ is therefore singular.
\end{remark}

\begin{remark}
\label{algremark3}
A direct calculation using the space presented in Theorem~\nwrefB{2.9}
shows that when $X$ is quasihypermetric but not strictly quasihypermetric
and has $\mbar(X) = \infty$,
it is possible to have the corresponding distance matrix~$D$ non-singular.
We present an example to show that $D$ may also be
singular under the same assumptions.

Let $X = \{x_1, x_2, x_3\}$ have the metric~$d_1$ with respect to which
all non-zero distances equal~$6$.
Then $X$ is clearly (strictly) quasihypermetric. Also, the measure
$\mu_1 = \frac{1}{3}(\delta_{x_1} + \delta_{x_2} + \delta_{x_3})$
is invariant on~$X$, and we have $\mbar(X) = 4$.
Let $Y = \{y_1, y_2, y_3, y_4\}$, where $y_1, y_2, y_3, y_4$
are equally spaced points placed consecutively around a circle
of radius~$\frac{4}{\pi}$, and give~$Y$ the arc-length metric~$d_2$.
Using the proof of Corollary~\nwrefB{2.4} and Example~\nwrefB{4ptexample},
we find that $Y$ is quasihypermetric but not strictly quasihypermetric,
that $\frac{1}{2}(\delta_{y_1} + \delta_{y_3})$ is invariant on~$Y$,
and that $\mbar(Y) = 2$.

Let $Z = X \cup Y$ and set $c = 3$.
Then defining $d \colon Z \times Z \to \R$ as in Theorem~\nwrefB{2.1},
we find that $Z$ is quasihypermetric, while $Z$ fails to be strictly quasihypermetric since $Y$ is not strictly quasihypermetric.
Further, by Theorem~\nwrefB{2.1.4},
the measure $-\mu_1 + \mu_2 \in \Mzero(Z)$
is invariant with value~$-1$, and it follows by Theorem~\ref{2.13new}
that $\mbar(Z) = \infty$. Finally, the distance matrix~$D$
for~$Z$ is singular, since its null space is the $1$-dimensional space
spanned by the vector $(0, 0, 0, -1, 1, -1, 1)^T$.
\end{remark}

Our next result gives a systematic account of the relationships
that must hold between the number of points in a finite space
and certain of the metric properties of the space.

First, we recall the following definition, due to Kelly~\cite{Kel2}.
Let $\Xd$ be a metric space. If for all $n \in \N$
and for all $a_1, \ldots, a_n, b_1, \ldots, b_{n+1} \in X$ we have
\[
\sum_{i=1}^n \sum_{j=1}^n d(a_i, a_j)
    + \sum_{i=1}^{n + 1} \sum_{j=1}^{n+1} d(b_i, b_j)
\leq
2 \sum_{i=1}^n \sum_{j=1}^{n+1} d(a_i, b_j),
\]
then $\Xd$ is said to be a \textit{hypermetric space}.

\begin{theorem}
\label{tabletheorem}
The following table gives the necessary relations between the
number of points in a finite space and various metric properties
of the space. \textup{(}An entry of a dash `---' should be read as
`sometimes yes and sometimes no'.\textup{)}

\vspace{2mm}
\begin{small}
\renewcommand{\arraystretch}{1.3}
\setlength{\tabcolsep}{4pt}
\begin{center}
\begin{upshape}
\begin{tabular}{|c||c|c|c|c|c|}
\hline
\parbox[t]{17mm}{\vspace*{-3mm}\center number of points in space}
 & \parbox[t]{17mm}%
   {\vspace*{-3mm}\center euclidean}
 & \parbox[t]{17mm}%
   {\vspace*{-3mm}\center hy\-per\-me\-t\-ric}
 & \parbox[t]{17mm}%
   {\vspace*{-3mm}\center quasi\-hy\-per\-me\-t\-ric}
 & \parbox[t]{17mm}%
   {\vspace*{-3mm}\center strictly quasi\-hy\-per\-me\-t\-ric\vspace*{2.5mm}}
 & \parbox[t]{17mm}%
   {\vspace*{-3mm}\center $\mbar < \infty$} \\ \hline\hline
$\llap{$\leq {}$} 3$ & yes  & yes  & yes  & yes  & yes \\[0.8mm] \hline
$4$                  & ---  & yes  & yes  & ---  & yes \\[0.8mm] \hline
$\llap{$\geq {}$} 5$ & ---  & ---  & ---  & ---  & --- \\[0.8mm] \hline
\end{tabular}
\end{upshape}
\end{center}
\end{small}
\end{theorem}

\begin{proof}
The following well established general results
(some of which we have already mentioned here or in~\cite{NW1})
deal with a number of cases immediately.
\begin{enumerate}
\item
By Theorem~3.8 of~\cite{AandS}, all compact subsets of euclidean spaces
have $\mbar$ finite.
\item
By Theorem~5.1 of~\cite {Kel1}, all euclidean spaces are hypermetric.
\item
By Theorem~2 of~\cite{Kel3}, all hypermetric spaces are quasihypermetric.
\item
By Lemma~1 of~\cite{Bjo}, all compact subsets of euclidean spaces are
strictly quasihypermetric. (The fact that finite subsets of
euclidean space are strictly quasihypermetric was proved in~\cite{Sch2}.)
\end{enumerate}

The only entries in the table now needing comment are disposed of
(with some redundancy) by the following observations.
\begin{enumerate}
\setcounter{enumi}{4}
\item
Every $4$-element metric space is $L^1$-embeddable, by~\cite{Wolfe}
(the authors are grateful to David Yost for pointing out
this fact and for locating the reference), and therefore hypermetric
(see Remark~\ref{Assouad} above).
(Blumenthal's four-point theorem \cite[Theorem~52.1]{Blu1}
shows independently that such a space is quasihypermetric.)
\item
Example~\nwrefB{4ptexample} gives a $4$-element metric space which
is not strictly quasihypermetric (but is hypermetric).
\item
We noted in~(5) that every $4$-element metric space is hypermetric,
and Remark~\ref{Assouad} outlines the argument
that the value of $\mbar$ must then be finite.
\item
Theorem~\nwrefB{2.9} constructs a $5$-element space which is
quasihypermetric but not strictly quasihypermetric and has $\mbar$ infinite.
\item
Assouad~\cite[Proposition~2]{Ass2} constructs a $5$-element metric space
which is quasihypermetric but not hypermetric
(further information is given in Example~\ref{Assouad1} below).
\item
Theorem~\nwrefB{2.9A} gives an example of a $5$-point space
which is not quasihypermetric. (The optimality of the number~$4$
in Blumenthal's four-point theorem also corresponds to the existence
of such a space.)
\end{enumerate}
\end{proof}

A natural question raised by the above results is
whether a strictly quasihypermetric metric space must be hypermetric.
We have seen in part~(5) of the proof of Theorem~\ref{tabletheorem}
that there is no $4$-point counterexample,
but we present one with $5$ points.

\begin{example}
Let $X = \{x_1, x_2\}$ and $Y = \{y_1, y_2, y_3\}$,
and give each set the discrete metric.
If we define~$Z$ as in Theorem~\nwrefB{2.1}, taking $c = \frac{5}{8}$,
then it follows that $Z$ is strictly quasihypermetric.
But taking $a_i = x_i$ for $i = 1, 2$ and $b_j = y_j$ for $j = 1, 2, 3$,
we find using the definition of Kelly above that $Z$ is not hypermetric.
\end{example}

\begin{example}
\label{Assouad1}
We show that the $5$-element space of Assouad referred to in part~(9)
of the proof of Theorem~\ref{tabletheorem} is not strictly quasihypermetric
and has $\mbar$ infinite.
The distances in this space are represented in the obvious way
by the entries of the following matrix:
\[
\left(
  \begin{array}{ccccc}
    0 & 2 & 2 & 5 & 5 \\
    2 & 0 & 4 & 3 & 3 \\
    2 & 4 & 0 & 3 & 3 \\
    5 & 3 & 3 & 0 & 4 \\
    5 & 3 & 3 & 4 & 0
  \end{array}
\right).
\]
It is easy to check that if we define a measure $\mu$ of mass~$0$ on the
space by using the matrix of respective weights
\[
\left(\!\!
  \begin{array}{r}
     2 \\
    -2 \\
    -2 \\
     1 \\
     1
  \end{array}
\right),
\]
then we have $d_\mu \equiv 2$, and the desired conclusions are given
by Theorem~\nwrefB{2.9.5new}.
\end{example}

\providecommand{\bysame}{\leavevmode\hbox to3em{\hrulefill}\thinspace}

\end{document}